\documentclass[a4paper,11pt]{article}
\usepackage{t1enc}
\usepackage[latin2]{inputenc}

\usepackage[margin=2cm]{geometry}
\usepackage{amsmath,amssymb,amsthm,indentfirst,eucal,latexsym}
\usepackage{hyperref}

\usepackage{newcent}

\newtheorem{theorem}{Theorem}[section]

\newtheorem{lemma}[theorem]{Lemma}

\newtheorem{definition}[theorem]{Definition}

\title{An explicit solution for optimal investment problems with autoregressive prices and
exponential utility\thanks{S. D. is with {E\"otv\"os Lor\'and University, Budapest}, M. R. is with
MTA Alfr\'ed R\'enyi Institute of Mathematics, Budapest. Part of this research was carried out while the second author was
also affiliated with the University of Edinburgh.}}

\author{S\'andor De\'ak \and 
Mikl\'os R\'asonyi}

\date{\today}

\begin{document}

\maketitle

\begin{abstract} 
We calculate explicitly the optimal strategy for an investor with exponential utility function when the stock price follows an autoregressive
Gaussian process. We also calculate its performance and analyse it when the trading horizon tends to infinity. Dependence of
asymptotic performance on the autoregression parameter is determined. This provides, to the best of our knowledge, the first instance
of a theorem linking directly the memory of the asset price process to the attainable satisfaction level of investors trading
in the given asset.
\end{abstract}

\textbf{MSC (2010):} Primary 93E20; Secondary
91G10.

\smallskip

\textbf{Keywords:} expected utility maximization; Gaussian autoregressive process; memory of a stochastic process

\smallskip

\section{Introduction}

Sequences of independent random variables have no memory at all, Markovian processes remember their past 
through their present value only.
In the case of processes with longer memory the entire past may influence the current evolution of the given
stochastic system, e.g. in the case of fractional Brownian motion and related processes.

Econometric time series exhibit various degrees of influence of the past on the present, 
depending on the sampling frequency. 
High-frequency
volatility has long-range dependence while asset prices may or may not have this property, \cite{baillie}.
The principal motivating question of our research is the following: how does the memory of an asset's price
influence the satisfaction attainable from investing into this asset ? 

The present paper concentrates on a Markovian setting. It 
precisely characterizes the dependence of performance on memory in a concrete model class where
the price follows a Gaussian autoregressive processes. In the case of investors with exponential utility we
find the optimal trading strategy for each finite time horizon and analyse what happens when the horizon tends
to infinity. We determine the exact dependence of the asymptotic performance on the autoregression parameter and
hence make the first step towards general results linking investment performance to the memory length 
of the underlying
security price.

The present paper continues previous investigations of \cite{dokuchaev,fs,martin}, where asymptotic
arbitrage in the utility sense was considered, i.e. the speed of the expected utility growth
when the time horizon tends to infinity. The first two references concentrated on
continuous-time models, \cite{martin} treated a model where borrowing and short-selling were
forbidden and utility functions were defined on the positive axis only.

The possibly negative prices of the model we consider may be acceptable in certain contexts (e.g. futures trading).
Its parameters may also be tuned such that negative prices practically never occur.
We nonetheless stress that our purpose is to exhibit a theoretical model whose qualitative conclusions are 
hoped to extend 
to a broader class of processes in the future so we are not bothered by the eventual negativity of prices.

We stress that it occurs very rarely that optimal strategies can be determined in closed form
for discrete-time investment problems. As far as we know our paper is the first to have found
the explicit solution for the case of autoregressive Gaussian processes.

In the present section we explain our model and the optimisation problem in consideration. In Section 2 we 
present our 
results,  Section 3 contains the proofs.

We are working with a financial market in which two assets are traded: a riskless asset with price constant $1$,
and a single risky asset whose price $X_t$ is an $\mathbb{R}$-valued stochastic process governed by 
the equation
\begin{equation}\label{80}
X_{t+1}=\alpha X_t +\sigma \varepsilon_{t+1},\ t\geq 0,
\end{equation}
where $\alpha\in\mathbb{R}$, $\sigma>0$ are parameters and $\varepsilon_t$ are i.i.d. 
standard Gaussian, independent of $X_0$. Introducing $\beta:=\alpha-1$, we may rewrite
\eqref{80} as
\begin{equation}
X_{t+1}-X_t=\beta X_t +\sigma\varepsilon_{t+1},\ t\geq 0.
\end{equation}
The information flow is given by
\begin{equation}
\mathcal{F}_t :=\sigma(X_s, 0\leq s \leq t).
\end{equation}
We interpret $\alpha$ (or, equivalently, $\beta$) as a ``memory parameter'' indicating how
previous values of the process $X$influence its present value. Eventually, our purpose is to
find the dependence of maximal achievable utility on this parameter.

A trading strategy is described by the number of units in the risky asset at $t$, denoted 
by $\phi_t$ for $t\geq 1$. Trading strategies are assumed $(\mathcal{F}_t)_{t\geq 0}$-predictable 
$\mathbb{R}$-valued 
processes (i.e. $\phi_t$ is $\mathcal{F}_{t-1}$-measurable for all $t$), in particular, short-selling is
allowed. The totality of trading strategies is denoted by $\Phi$.

The wealth process corresponding to a given trading strategy $(\phi_t)_{t\geq 1}$ is 
\begin{equation}\label{logwealth}
L_t^{\phi}:=L_{t-1}^{\phi}+\phi_t(X_t-X_{t-1}), \qquad t\geq 1,
\end{equation}
where $L_0^{\phi}:=L_0$ is the initial capital of the investor.
In other words, the terminal wealth of the investor is given by 
\begin{equation}\label{logwealth2}
L_T^{\phi}=L_0+\sum_{j=1}^T\phi_j(X_j-X_{j-1}),
\end{equation}
where $T\geq 1$ is a time horizon.

We focus on a finite horizon utility maximization problem and look for the optimal strategy $({\phi}_t^*)_{1\leq t\leq T}$ 
which satisfies
\begin{equation}\label{90}
\sup_{\phi\in\Phi}\mathbb{E}U\left(L_T^{\phi}\right)=\mathbb{E}U\left(L_T^{{\phi}^*}\right),
\end{equation}
where $U:\mathbb{R}\rightarrow\mathbb{R}$ is the utility function $U(x)=-e^{-x}$. Note that 
the expectations exist but may be $-\infty$.
We are going to give an explicit solution for this problem.

In order to make a comparison, we also consider an investor who is not using the accumulated 
past information,
i.e. we define $\Phi_0$ as the set of trading strategies for which $\phi_t$ is $\mathcal{F}_0$-measurable for 
all $1\leq t\leq T$. We wish to find ${\eta}^*\in\Phi_0$ such that
$$
\sup_{\phi\in\Phi_0}\mathbb{E}U\left(L_T^{\phi}\right)=\mathbb{E}U\left(L_T^{{\eta}^*}\right).
$$
We may and will suppose $L_0=0$ in the sequel.

\section{Results}

\begin{theorem}\label{maintheorem}
\begin{enumerate}
	\item  The optimal strategies for time horizon $T\in\mathbb{N}$ are
				$({\phi}^*_t)_{1\leq t\leq T}:=(\hat{\phi}_t^T(X_{t-1}))_{1\leq t\leq T}$ when past information
				is used and
				$({\eta}^*_t)_{1\leq t\leq T}:=(\hat{\phi}_t^T(X_{0}))_{1\leq t\leq T}$ when past 
				information is ignored 
				where
				\begin{equation}\label{strategies}
					\hat{\phi}_t^T(z)=\frac{\beta z}{\sigma^2}\theta_t^T \qquad \textrm{for all }1\leq t\leq T\mbox{ and }z\in\mathbb{R},
				\end{equation}
				and  $\theta_t^T=1-(T-t)\beta$.
	\item Using these strategies, the maximal conditional expected utilities are 
				\begin{equation}\label{condexp}
				\mathbb{E}\left[U\left(L_T^{{\phi}^*}\right)|X_0=z\right]=-\frac{1}{\sqrt{\gamma_\beta(T)}}e^{-\frac{\beta^2z^2}{2\sigma^2}T} 
						\textrm{, and}
				\end{equation}
				and
				\begin{equation}\label{condexpnomem}
					\mathbb{E}\left[U\left(L_T^{{\eta}^*}\right)|X_0=z\right]=-e^{-\frac{\beta^2z^2}{2\sigma^2}T},
				\end{equation}
				respectively, where $\gamma_\beta$ is given by
				\begin{equation}\label{gamma}
					\gamma_\beta(T)=\left\{\begin{array}{cl} 
					\frac{\beta^{2T}\Gamma(\frac{1}{\beta^2}+T)}{\Gamma(\frac{1}{\beta^2})} & \textrm{if }\beta\neq 0, \\
					 1 & \textrm{if }\beta=0, \end{array}\right. \\
				\end{equation}
				and $\Gamma$ is the well-known gamma function, $\Gamma(x):=\int_0^{+\infty}t^{x-1}e^{-t}\mathrm{d}t$.
	\item In case of a stable autoregressive process (when $|\alpha|<1$, or, equivalently, $\beta\in (-2,0)$) 
		assuming that $\mathrm{var}(X_t)=1$, and $X_0$ is $N(0,1)$, the maximal expected utility is
				\begin{equation}\label{exp}
					\mathbb{E}\left[U\left(L_T^{{\phi}^*}\right)\right]=-
					\sqrt{\frac{\beta+2}{\left(2-(T-1)\beta\right)\gamma_{\beta}(T)}}\textrm{, and}
				\end{equation}
				\begin{equation}\label{expnomem}
					\mathbb{E}\left[U\left(L_T^{{\eta}^*}\right)\right]=-\sqrt{\frac{\beta+2}{\left(2-(T-1)\beta\right)}}.
				\end{equation}
	\item The asymptotic behaviour of $\gamma_\beta$ is 
				$\lim_{T\rightarrow+\infty}\frac{\gamma_\beta(T)}{h_\beta(T)}=1$
				where
				\begin{equation}\label{asymp}
					h_\beta(T)=\left\{\begin{array}{cl} 
					\Gamma\left(\frac{1}{\beta^2}\right)\left(\beta^2\right)^{1-\frac{1}{\beta^2}}\sqrt{2\pi\left(T-1+\frac{1}{\beta^2}\right)} 
					\left(\frac{1+(T-1)\beta^2}{e}\right)^{T-1+\frac{1}{\beta^2}} & \textrm{if }
					\beta\neq 0, \\ 1 & \textrm{if }\beta=0. \end{array}\right. \\
				\end{equation}
	\end{enumerate}
\end{theorem}
 
Our conclusion is that using accumulated past information leads to a qualitatively better strategy
(i.e. the expected utility using past information is the expected utility without past information
times a factor tending to $0$ much faster than that). 
In order to make a meaningful comparison, however, we should normalize our processes.
Hence in 3. above we assume stability and choose $X_0$ in such a way that
$X_t$, $t\in\mathbb{N}$ becomes stationary with $\mathbb{E}X_t^2=1$. 
We obtain the maximal expected utilities
both with and without past information.
They show that an increase in $\beta$ leads to an increase in \eqref{exp} and \eqref{expnomem}, 
so more information  from the
past (represented by 
lager $|\beta|$) leads to better performance. This platitude is, however, supported by precise formulae in the present case.
Note also that $\gamma_{\beta}$ is continuous at $\beta=0$ as well.

\section{Computations and proofs}



The following lemma will help to construct the optimal strategy $\phi^*$ inductively.

\begin{lemma}\label{argumentofphistatement} Let $\tilde{\phi}_t$, $t=1,\ldots,T-1$ be a trading
strategy up to time $T-1$ given by $\tilde{\phi}_t=f_t(X_0,\ldots, X_{t-1})$ with Borel functions $f_t$. If $\tilde{\phi}$ is optimal up to $T-1$, i.e.
for all strategies $\phi_t$, $t=1,\ldots,T-1$ one has 
\begin{equation}\label{porous}
\mathbb{E}\left[U(L^{\phi}_{T-1})\vert\mathcal{F}_{0}\right]\leq \mathbb{E}\left[U(L^{\tilde{\phi}}_{T-1})\vert\mathcal{F}_{0}\right],
\end{equation}
then the strategy defined by $\bar{\phi}_t:=f_{t-1}(X_1,\ldots,X_{t-1})$, $t=2,\ldots,T$ is optimal between $1$ and $T$, i.e. for all
strategies $\psi_t$, $t=2,\ldots,T$ one has
\begin{equation}\label{pp}
\mathbb{E}\left[U\left(\sum_{t=2}^T {\psi}_t (X_t-X_{t-1})\right)\vert\mathcal{F}_{1}\right]\leq \mathbb{E}\left[U\left(\sum_{t=2}^T \bar{\phi}_t (X_t-X_{t-1})\right)\vert\mathcal{F}_{1}\right].
\end{equation}
\end{lemma}
\begin{proof} 
Let us denote by $\ell(dy_T,\ldots,dy_2|y_1,y_0)$ the conditional law of $(X_T,\ldots,X_2)$ w.r.t. 
to $X_1=y_1$, $X_0=y_0$. We fix a regular version (see III. 70--73 in \cite{dm}).
As $X$ is a homogeneous Markov chain, $\ell$ does not depend on $y_0$ (hence we will write, with a slight abuse of notation, $\ell(dy_T,\ldots,dy_2|y_1)$)
and $\ell(dy_{T-1},\ldots,dy_1|y_0)$ is also the density function of  $(X_{T-1},\ldots,X_1)$ conditional to $X_0=y_0$

Let $\psi_t=g_t(X_0,\ldots,X_{t-1})$, $t=2,\ldots,T$ with some Borel functions $g_t$. Define for each $z\in\mathbb{R}$
the strategy $\phi_t:=\phi_t(z)=g_{t+1}(z,X_0,\ldots,X_{t-1})$, $t=1,\ldots,T-1$. By \eqref{porous},
\begin{eqnarray}\label{ppp}
\int_{\mathbb{R}^{T-1}}U\left(\sum_{t=1}^{T-1} g_{t+1}(z,y_0,\ldots,y_{t-1})(y_t-y_{t-1})\right) 
\ell(dy_{T-1},\ldots,dy_1|y_0)
\leq\\
\nonumber
\int_{\mathbb{R}^{T-1}}U\left(\sum_{t=1}^{T-1} f_t(y_0,\ldots,y_{t-1})(y_t-y_{t-1})\right) \ell(dy_{T-1},\ldots,dy_1|y_0),
\end{eqnarray}
for a.e. $y_0$ (with respect to the law of $X_0$). Note that the right-hand side can be rewritten as
\begin{eqnarray}\label{pppp}
\int_{\mathbb{R}^{T-1}}U\left(\sum_{t=2}^{T} f_{t-1}(y_1,\ldots,y_{t-1})(y_t-y_{t-1})\right) \ell(dy_{T},\ldots,dy_2|y_1) 
=\\ \nonumber
\mathbb{E}\left[U\left(\sum_{t=2}^T \bar{\phi}_t (X_t-X_{t-1})\right)|X_1=y_1,X_0=y_0\right],
\end{eqnarray}
for a.e. $(y_1,y_0)$ w.r.t. the law of $(X_0,X_1)$, note the Markov property again. 
Similarly, the left-hand side of \eqref{ppp} equals
\begin{eqnarray}\label{ppppp}
\int_{\mathbb{R}^{T-1}}U\left(\sum_{t=2}^{T-2} g_{t}(z,y_1,\ldots,y_{t-1})(y_t-y_{t-1})\right) \ell(dy_{T},\ldots,dy_2|y_1) =\\ \nonumber \mathbb{E}\left[U\left(\sum_{t=2}^T {\psi}_t (X_t-X_{t-1})\right)|X_1=y_1,X_0=z\right]. 
\end{eqnarray}
Now plugging in $z:=y_0$ we obtain \eqref{pp} from \eqref{ppp}, \eqref{pppp} and \eqref{ppppp}. 
\end{proof}

After these preparations we are able to give an explicit solution for the optimal strategies of the wealth process in case of the price is an autoregressive process.

In this Section we prove \autoref{maintheorem}, first we focus on the case where the investor 
uses past information. We consider the case $T=1$, so the wealth process according to 
\eqref{logwealth} takes the form
\begin{equation}
L_1^{\phi}=\phi_1(X_1-X_0)=\phi_1(\beta X_0+\sigma\varepsilon_1).
\end{equation}

We have
\begin{equation}
	\mathbb{E}\left[e^{-\phi_1(\beta X_0+\sigma \varepsilon _1)} | X_0=z\right]
	= e^{-\phi_1\beta z} \mathbb{E}\left[e^{-\phi_1 \sigma \varepsilon_1}\right]=
	e^{-\phi_1\beta z}e^{\frac{\phi_1^2\sigma^2}{2}},
\end{equation}
hence we get
\[
\mathrm{arg}\min _{\phi_1}e^{-\phi_1\beta z+\frac{(\phi_1\sigma)^2}{2}}
 = \mathrm{arg}\min_{\phi_1}e^{\frac{1}{2}\left(\phi_1\sigma-\frac{\beta z}{\sigma}\right)^2-\frac{\beta^2 z^2}{2\sigma^2}}
 =\frac{\beta z}{\sigma^2}=\frac{\beta z}{\sigma^2}\theta_1^1=\hat{\phi}_1^1(z), \]
 because $\theta_1^1=1$. So we proved the first part of \autoref{maintheorem} for $T=1$. Now let's assume that  
 \eqref{strategies} is true for $T-1$, i.e.
 \begin{equation}\label{opp}
\phi^*_t:=\hat{\phi}_t^{T-1}(X_{t-1})\mbox{ with }\hat{\phi}_t^{T-1}(z)=\frac{\beta z}{\sigma^2}\theta_t^{T-1} \quad 1\leq t\leq T-1
 \end{equation}
satisfies \eqref{90} for all $\phi\in\Phi$. 
 We will prove that \eqref{strategies} also holds for $T$. By Lemma \ref{argumentofphistatement},
 for all $\psi\in\Phi$,
 \begin{eqnarray}\nonumber
\mathbb{E}[e^{-L^{\psi}_T}|\mathcal{F}_0] = \mathbb{E}[e^{-\psi_1(X_1-X_0)}
\mathbb{E}[e^{-\sum_{j=2}^T \psi_j(X_j-X_{j-1})}|\mathcal{F}_1]\vert\mathcal{F}_0] &\geq &\\ \nonumber
\mathbb{E}[e^{-\psi_1(X_1-X_0)}
\mathbb{E}[e^{-\sum_{j=2}^T \hat{\phi}_{j-1}^{T-1}(X_{j-1})(X_j-X_{j-1})}|\mathcal{F}_1]\vert\mathcal{F}_0]&=&\\
\mathbb{E}[e^{-\psi_1(X_1-X_0)}
\mathbb{E}[e^{-\sum_{j=2}^T \hat{\phi}_{j}^{T}(X_{j-1})(X_j-X_{j-1})}|\mathcal{F}_1]\vert\mathcal{F}_0],\label{parr} & &
 \end{eqnarray}

 
since $\phi_{j-1}^{T-1}=\phi_j^T$. Now define the trading strategy 
$\omega=(\phi, \hat{\phi}_2^T(X_1), \dots , \hat{\phi_T^T}(X_{T-1}))$ and the function 
$Q_T\,:\,\mathbb{R}^{T+2} \rightarrow\mathbb{R}$ such that
\begin{equation}
Q_T(\phi,X_0,\mathbf{\varepsilon}):=L_T^{\omega},\quad\textrm{where }\mathbf{\varepsilon}=(\varepsilon_1, \dots \varepsilon_T)^{\mathbf{T}}.
\end{equation}
Hence, according to \eqref{parr}, it remains to find $\phi$ which minimizes
\begin{equation}\label{phi1def}
\mathbb{E}\left[e^{-Q_T(\phi,X_0,\mathbf{\varepsilon})}|X_0=z\right].
\end{equation}
If we prove that $\phi=\hat{\phi}_1^{T}(z)$ does the job then we will be able to conclude that the optimal
strategy for time horizon $T$ is indeed as given in \eqref{opp} for $T-1$.


To compute the minimiser $\phi$ we will write $Q_T(\phi, X_0, \varepsilon)$ 
in a sum of a quadratic, a linear and a constant function of $\mathbf{\varepsilon}$.

\begin{eqnarray*}
 	Q_T(\phi,X_0,\mathbf{\varepsilon})&=&\phi(X_1-X_0)+\sum_{j=2}^T\hat{\phi}_j^T(X_{j-1})(X_j-X_{j-1}) \\
 	   &=&\phi(X_1-X_0)+\sum_{j=2}^T\frac{\beta\theta_j^T}{\sigma^2}\underbrace{X_{j-1}(X_j-X_{j-1})}_{A_j}
\end{eqnarray*}

For $A_j$ we have
\begin{eqnarray*}
 	A_j&=&\left(\alpha^{j-1}X_0+\sigma\sum_{i=1}^{j-1}\alpha^{j-i-1}\varepsilon_i\right)\left(\alpha^{j}X_0+
 		\sigma\sum_{i=1}^{j}\alpha^{j-i}\varepsilon_i - \alpha^{j-1}X_0-\sigma\sum_{i=1}^{j-1}\alpha^{j-i-1}\varepsilon_i\right) \\
 	&=&\left(\alpha^{j-1}X_0+\sigma\sum_{i=1}^{j-1}\alpha^{j-i-1}\varepsilon_i\right)\left(\alpha^{j-1}\beta X_0+\sigma\varepsilon_j +
 		\sigma\sum_{i=1}^{j-1}\alpha^{j-i-1}\beta\varepsilon_i\right) \\
 	&=&\alpha^{2j-2}\beta X_0^2+\sigma\beta X_0\sum_{i=1}^{j-1}\alpha^{2j-i-2}\varepsilon_i+\sigma X_0\alpha^{j-1}\varepsilon_j+ 
 		\sigma^2\sum_{i=1}^{j-1}\alpha^{j-i-1}\varepsilon_i\varepsilon_j+\sigma\beta X_0\sum_{i=1}^{j-1}\alpha^{2j-i-2}\varepsilon_i \\
 	&\,& +\sigma^2\beta\left(\sum_{i=1}^{j-1}\alpha^{j-i-1}\varepsilon_i\right)^2 \\
 	&=&\underbrace{\alpha^{2j-2}\beta X_0^2}_{B_1(j)}+ \underbrace{2\sigma\beta X_0\sum_{i=1}^{j-1}\alpha^{2j-i-2}\varepsilon_i}_{B_2(j)} 
 		+\underbrace{\sigma X_0\alpha^{j-1}\varepsilon_j}_{B_3(j)}
 		+ \underbrace{\sigma^2\sum_{i=1}^{j-1}\alpha^{j-i-1}\varepsilon_i\varepsilon_j}_{B_4(j)} 
 		+\underbrace{\sigma^2\beta\sum_{i=1}^{j-1}\sum_{k=1}^{j-1}\varepsilon_i\varepsilon_k\alpha^{2j-i-k-2}}_{B_5(j)}. \\
\end{eqnarray*}

Let's substitute this into  $Q_T(\phi,X_0,\mathbf{\varepsilon})$.
\begin{eqnarray*}
 	Q_T(\phi,X_0,\mathbf{\varepsilon})&=&\phi(X_1-X_0)+\sum_{j=2}^T\frac{\beta\theta_j^T}{\sigma^2}(B_1(j)+B_2(j)+B_3(j)+B_4(j)+B_5(j)) \\
 		 &=&\phi(\beta X_0+\sigma\varepsilon_1)+\underbrace{\sum_{j=2}^T\frac{\beta\theta_j^T}{\sigma^2}B_1(j)}_{C_1}
 		 	+ \underbrace{\sum_{j=2}^T\frac{\beta\theta_j^T}{\sigma^2}B_2(j)}_{C_2} 
 		 	+\underbrace{\sum_{j=2}^T\frac{\beta\theta_j^T}{\sigma^2}B_3(j)}_{C_3}+ \\
 		 &+&\underbrace{\sum_{j=2}^T\frac{\beta\theta_j^T}{\sigma^2}B_4(j)}_{C_4} 
 		 	+\underbrace{\sum_{j=2}^T\frac{\beta\theta_j^T}{\sigma^2}B_5(j)}_{C_5}
\end{eqnarray*}
 	
We compute each $C_n$ separately. 
\begin{eqnarray*}
 	C_1&=&\frac{\beta^2 X_0^2}{\sigma^2}\sum_{j=2}^T\theta_j^T\alpha^{2j-2}. \\
 	C_2&=&\sum_{j=2}^T\frac{\beta\theta_j^T}{\sigma^2}2\sigma\beta X_0\sum_{i=1}^{j-1}\alpha^{2j-i-2}\varepsilon_i
 	   =\sum_{j=1}^{T-1}\sum_{i=1}^j\frac{2\beta^2X_0}{\sigma}\theta_{j+1}^T\alpha^{2j-i}\varepsilon_i \\
 	   &=&\sum_{i=1}^{T-1}\sum_{j=i}^{T-1}\frac{2\beta^2X_0}{\sigma}\theta_{j+1}^T\alpha^{2j-i}\varepsilon_i
 	   =\sum_{i=1}^{T-1}\varepsilon_i\left(\sum_{j=i}^{T-1}\frac{2\beta^2X_0}{\sigma}\theta_{j+1}^T\alpha^{2j-i}\right).\\
 	C_3&=&\sum_{j=2}^T\frac{\beta\theta_j^T}{\sigma^2}\sigma X_0\alpha^{j-1}\varepsilon_j
 		 =\sum_{j=2}^T\varepsilon_j\left(\frac{\beta X_0}{\sigma}\theta_{j}^T\alpha^{j-1}\right).\\
 	C_4&=&\sum_{j=2}^T\frac{\beta\theta_j^T}{\sigma^2}\sigma^2\sum_{i=1}^{j-1}\alpha^{j-i-1}\varepsilon_i\varepsilon_j
 		 =\sum_{j=2}^T\sum_{i=1}^{j-1}\varepsilon_i\varepsilon_j\left(\beta\theta_j^T\alpha^{j-i-1}\right).\\	C_5&=&\sum_{j=2}^T\frac{\beta\theta_j^T}{\sigma^2}\sigma^2\beta\sum_{i=1}^{j-1}\sum_{k=1}^{j-1}\varepsilon_i\varepsilon_k\alpha^{2j-i-k-2}
 		  = \beta^2\sum_{j=1}^{T-1}\sum_{i=1}^j\sum_{k=1}^j\theta_{j+1}^T\varepsilon_i\varepsilon_k\alpha^{2j-i-k} \\
 	&=&\beta^2\sum_{i=1}^{T-1}\sum_{j=1}^{T-1}\sum_{k=1}^j\theta_{j+1}^T\varepsilon_i\varepsilon_k\alpha^{2j-i-k} 
 		= \beta^2\sum_{i=1}^{T-1}\left(\sum_{k=1}^{i}\sum_{j=i}^{T-1}\theta_{j+1}^T\varepsilon_i\varepsilon_k\alpha^{2j-i-k}
 		+ \sum_{k=i+1}^{T-1}\sum_{j=k}^{T-1}\theta_{j+1}^T\varepsilon_i\varepsilon_k\alpha^{2j-i-k}\right) \\
 	&=&\beta^2\sum_{i=1}^{T-1}\left(\sum_{k=1}^{i}\varepsilon_i\varepsilon_k\sum_{j=i}^{T-1}\theta_{j+1}^T\alpha^{2j-i-k}
 		+ \sum_{k=i+1}^{T-1}\varepsilon_i\varepsilon_k\sum_{j=k}^{T-1}\theta_{j+1}^T\alpha^{2j-i-k}\right).
\end{eqnarray*}
 		
According to these, we can write $Q_T(\phi,X_0,\mathbf{\varepsilon})$ as
\begin{equation}\label{quadraticform}
 		Q_T(\phi,X_0,\mathbf{\varepsilon})=
 		\mathbf{\varepsilon^T}\left(\mathbf{A}_T-\frac{1}{2}\mathbf{I}\right)\mathbf{\varepsilon}+\mathbf{b^T}(\phi,X_0)\mathbf{\varepsilon}+c(\phi,X_0),
\end{equation}
 	
 	where $\mathbf{A}_T=[a_{ik}]\in\mathbb{R}^{T\times T}$ is a symmetric matrix with
 	\begin{eqnarray*} a_{ii}&=&\frac{1}{2}+\beta^2\sum_{j=i}^{T-1}\theta_{j+1}^T\alpha^{2j-2i}, \qquad 1\leq i\leq T-1, \\
 	 a_{TT}&=&\frac{1}{2}, \\
 	 a_{ik}&=&\frac{\beta\theta_i^T\alpha^{i-k-1}}{2}+\beta^2\sum_{j=i}^{T-1}\theta_{j+1}^T\alpha^{2j-i-k}, \qquad 1 \leq k<i\leq T-1,\\
 	 a_{Tk}&=&\frac{\beta\alpha^{T-k-1}}{2} \qquad 1\leq k\leq T-1;
 	\end{eqnarray*}
 	
 	 $\mathbf{b}\,:\,\mathbb{R}^2\rightarrow\mathbb{R}^T$ is a vector with
 	\begin{eqnarray*} b_1(\phi, X_0)&=&\phi\sigma+\frac{2\beta^2X_0}{\sigma}\sum_{j=1}^{T-1}\theta_{j+1}^T\alpha^{2j-1}, \\
 		 b_i(\phi, X_0)&=&\frac{\beta X_0}{\sigma}\theta_i^T\alpha^{i-1}+
 		 \frac{2\beta^2X_0}{\sigma}\sum_{j=i}^{T-1}\theta_{j+1}^T\alpha^{2j-i}, \qquad 2\leq i\leq T-1, \\
 		 b_T&=&\frac{\beta X_0}{\sigma}\alpha^{T-1};
 		 \end{eqnarray*}
 		 	
 	and $c\,:\,\mathbb{R}^2\rightarrow\mathbb{R}$ where
 	\[c(\phi,X_0)=\phi\beta X_0+\frac{\beta^2 X_0^2}{\sigma^2}\sum_{j=2}^T\theta_j^T\alpha^{2j-2}.\]
 	
 We need to compute the conditional expected utility given by
 \begin{equation}\label{integralformofexp}
 \mathbb{E}\left[e^{-Q_T(\phi,X_0,\mathbf{\varepsilon})}|X_0=z\right]= \frac{1}{\left(\sqrt{2\pi}\right)^T}\int_{\mathbb{R}^T}e^{-\mathbf{x^TA}_T\mathbf{x}-\mathbf{b^T}(\phi,z)\mathbf{x}-c(\phi,z)} \mathrm{d}\mathbf{x}.
 \end{equation}
 In order to evaluate this integral we need some preparation.

 
 We know that
 \begin{equation}\label{integral1} 
 	\frac{1}{\sqrt{2\pi}}\int_{\mathbb{R}}e^{-ax^2-bx}\mathrm{d}x=\frac{1}{\sqrt{2a}}e^{\frac{b^2}{4a}},
 \end{equation}
for all  $b\in \mathbb{R}$ and $a>0$.

\begin {lemma}\label{integral2lemma}
Let $\mathbf{D}$ be a positive 
definite diagonal matrix with diagonal entries $d_1, d_2, \dots, d_n$, and let $\mathbf{b}\in\mathbb{R}^n$. Then
\begin{equation}\label{integral2}
	\frac{1}{\left(\sqrt{2\pi}\right)^n}\int_{\mathbb{R}^n}e^{-\mathbf{x^TDx}-\mathbf{b^Tx}}\mathrm{d}\mathbf{x}
	=\frac{1}{\sqrt{2^n\det \mathbf{D}}}e^{\frac{\mathbf{b^TD^{-1}b}}{4}}.
\end{equation}
\end{lemma}

\begin{proof} Using \eqref{integral1}
\begin{eqnarray*}
\frac{1}{\left(\sqrt{2\pi}\right)^n}\int_{\mathbb{R}^n}e^{-\mathbf{x^TDx}-\mathbf{b^Tx}}\mathrm{d}\mathbf{x}&=&
\frac{1}{\left(\sqrt{2\pi}\right)^n}\int_{\mathbb{R}^n}e^{-\sum_{i=1}^nd_ix_i^2-\sum_{i=1}^nb_ix_i}\mathrm{d}\mathbf{x}\\
&=&\prod_{i=1}^n\frac{1}{\sqrt{2\pi}}\int_{\mathbb{R}}e^{-d_ix_i^2-b_ix_i}\mathrm{d}x_i \\
&=& \prod_{i=1}^n\frac{1}{\sqrt{2d_i}}e^{\frac{b_i^2}{4d_i}}
 =\frac{1}{\sqrt{2^n\det \mathbf{D}}}e^{\frac{\mathbf{b^TD^{-1}b}}{4}}.
\end{eqnarray*}
\end{proof}

\begin{lemma}\label{integral3lemma}
Let $\mathbf{A}\in\mathbb{R}^{n\times n}$ be a symmetric, positive definite matrix , and $\mathbf{b}\in\mathbb{R}^n$. Then
\begin{equation}\label{integral3}
\frac{1}{\left(\sqrt{2\pi}\right)^n}\int_{\mathbb{R}^n}e^{-\mathbf{x^TAx}-\mathbf{b^Tx}}\mathrm{d}\mathbf{x} = \frac{1}{\sqrt{2^n\det \mathbf{A}}}e^{\frac{\mathbf{b^TA^{-1}b}}{4}}.
\end{equation}
\end{lemma}

\begin{proof} Since $\mathbf{A}$ is symmetric, there is an $\mathbf{S}$ orthonormal, and a $\mathbf{D}$ diagonal matrix for which $\mathbf{SDS^{-1}}=\mathbf{SDS^{T}}=\mathbf{A}$ and $|\det \mathbf{S}|=1$.
Using Lemma \ref{integral2lemma} and setting $\mathbf{y}:=\mathbf{{S^Tx}}$
\begin{eqnarray*}
\frac{1}{\left(\sqrt{2\pi}\right)^n}\int_{\mathbb{R}^n}e^{-\mathbf{x^TAx}-\mathbf{b^Tx}}\mathrm{d}\mathbf{x}
	&=& \frac{1}{\left(\sqrt{2\pi}\right)^n}\int_{\mathbb{R}^n}e^{-\mathbf{x^TSDS^Tx}-\mathbf{b^Tx}}\mathrm{d}\mathbf{x}
	=\frac{1}{\left(\sqrt{2\pi}\right)^n}\int_{\mathbb{R}^n}e^{-\mathbf{(S^Tx)^TDS^Tx}-\mathbf{b^Tx}}\mathrm{d}\mathbf{x} \\
	&=&\frac{1}{\left(\sqrt{2\pi}\right)^n}\int_{\mathbb{R}^n}e^{-\mathbf{y^TDy}-\mathbf{b^TSy}}|\det \mathbf{S}|\mathrm{d}\mathbf{y}
	=\frac{1}{\left(\sqrt{2\pi}\right)^n}\int_{\mathbb{R}^n}e^{-\mathbf{y^TDy}-\mathbf{(S^Tb)^Ty}}\mathrm{d}\mathbf{y}\\
	&=&\frac{1}{\sqrt{2^n\det \mathbf{D}}}e^{\frac{\mathbf{(S^Tb)^TD^{-1}S^Tb}}{4}}
	=\frac{1}{\sqrt{2^n\det\mathbf{D}}}e^{\frac{\mathbf{b^TSD^{-1}S^Tb}}{4}}\\
	&=&\frac{1}{\sqrt{2^n\det\mathbf{A}}}e^{\frac{\mathbf{b^TA^{-1}b}}{4}},
\end{eqnarray*}
since $\det\mathbf{D}=\det\mathbf{A}$ and $\mathbf{A^{-1}}=\mathbf{SD^{-1}S^T}$.
\end{proof}

Now we can compute the expression in \eqref{integralformofexp} using Lemma \ref{integral3lemma}:
\begin{equation}\label{matrixformofexp}
\mathbb{E}\left[e^{-Q_T(\phi,X_0,\mathbf{\varepsilon})}|X_0=z\right]
=\frac{1}{\sqrt{2^T\det \mathbf{A}_T}}e^{\frac{\mathbf{b^T}(\phi,z)\mathbf{A}_T\mathbf{^{-1}b}(\phi,z)}{4}-c(\phi,z)}.
\end{equation}

We proceed to examining the determinant of $\mathbf{A}_T$ 
to prove that $\mathbf{A}_T$ is positive definite (as \eqref{matrixformofexp} holds only in this case)
and we will need to compute one element of the inverse matrix, $\left(\mathbf{A^{-1}}\right)_{1,1}$. 

 
 First we present a lemma which will be very useful later.
 \begin{lemma} \label{simplelemma}
 For $\theta_t^T$ (defined in \autoref{maintheorem}) and for all $m\leq n$
 \begin{equation}
 \sum_{i=m}^n\theta_i^T\alpha^i=(T+1-m)\alpha^m-(T-n)\alpha^{n+1}.
 \end{equation}
\end{lemma}

\begin{proof} 
\begin{eqnarray*}
\sum_{i=m}^n\theta_i^T\alpha^i&=&\sum_{i=m}^n(T+1-i-(T-i)\alpha)\alpha^i=\sum_{i=m}^n(T+1-i)\alpha^i-\sum_{i=m}^n(T-i)\alpha^{i+1} \\
		&=&\sum_{i=m}^n(T+1-i)\alpha^i-\sum_{i=m+1}^{n+1}(T+1-i)\alpha^i=(T+1-m)\alpha^m-(T-n)\alpha^{n+1}
	\end{eqnarray*}
\end{proof}

\begin{lemma}\label{sumofrowsstatement}
 For $\mathbf{A}_T=[a_{ij}]$ we have
 \begin{equation}\label{sumofrows}
 a_{1,n}-\beta\sum_{i=2}^Ta_{in}=0\quad \textrm{for all }2\leq n\leq T. \end{equation}
 \end{lemma}
 
\begin{proof} First we consider the case $n=T$,
 \begin{eqnarray*}
 		a_{1,T}-\beta\sum_{i=2}^Ta_{i,T}&=&\frac{\beta\alpha^{T-2}}{2}
 			-\beta\left(\frac{1}{2}+\sum_{i=2}^{T-1}\frac{\beta\alpha^{T-i-1}}{2}\right)=
 			\alpha^{T-2}-1-\beta\sum_{i=2}^{T-1}\alpha^{T-i-1} \\
 		&=&\alpha^{T-2}-1-\sum_{i=2}^{T-1}\alpha^{T-i}+\sum_{i=2}^{T-1}\alpha^{T-i-1}=
 		\alpha^{T-2}-1-\sum_{i=2}^{T-1}\alpha^{T-i}+\sum_{i=3}^{T}\alpha^{T-i} \\
 		&=&\alpha^{T-2}-1-\alpha^{T-2}+1=0.
 		\end{eqnarray*}
 	Then we consider the case $n\neq T$,
 	\begin{equation}\label{sumofrows2}
 	a_{1,n}-\beta\sum_{i=2}^Ta_{i,n}=a_{1,n}-\beta\sum_{i=2}^{n-1}a_{i,n}-\beta a_{n,n}-\beta\sum_{i=n+1}^{T-1}a_{i,n}-\beta a_{T,n}.
 	\end{equation}
 	We compute the sums separately.
 	\begin{eqnarray*}
 	\beta\sum_{i=2}^{n-1}a_{i,n}&=&\sum_{i=2}^{n-1}\left(\frac{\beta^2\theta_n^T\alpha^{n-i-1}}{2}
 		+\beta^3\sum_{j=n}^{T-1}\theta_{j+1}^T\alpha^{2j-n-i}\right) \\
 	&=&\frac{\beta^2\theta_n^T}{2}\sum_{i=2}^{n-1}\alpha^{n-i-1}+\beta^3\sum_{j=n}^{T-1}\theta_{j+1}^T\sum_{i=2}^{n-1}\alpha^{2j-n-i} \\
 	&=&\frac{\beta\theta_n^T}{2}\left(\sum_{i=2}^{n-1}\alpha^{n-i}-\sum_{i=2}^{n-1}\alpha^{n-i-1}\right)
 		+ \beta^2\sum_{j=n}^{T-1}\theta_{j+1}^T\left(\sum_{i=2}^{n-1}\alpha^{2j-n-i+1}-\sum_{i=2}^{n-1}\alpha^{2j-n-i}\right) \\
 	&=&\frac{\beta\theta_n^T}{2}\left(\alpha^{n-2}-1\right)
 		+\beta^2\sum_{j=n}^{T-1}\theta_{j+1}^T \left(\alpha^{2j-n-1}-\alpha^{2j-2n+1}\right) \\
 	&=&\frac{\beta\theta_n^T\alpha^{n-2}}{2}-\frac{\beta\theta_n^T}{2}+\beta^2\sum_{j=n}^{T-1}\theta_{j+1}^T\alpha^{2j-n-1}
 		-\beta^2\sum_{j=n}^{T-1}\theta_{j+1}^T\alpha^{2j-2n+1} \\
\beta\sum_{i=n+1}^{T-1}a_{in}&=&\sum_{i=n+1}^{T-1}\left(\frac{\beta^2\theta_i^T\alpha^{i-n-1}}{2} 
		+\beta^3\sum_{j=i}^{T-1}\theta_{j+1}^T\alpha^{2j-i-n}\right)\\
 	&=&\frac{\beta^2\alpha^{-n-1}}{2}\sum_{i=n+1}^{T-1}\theta_i^T\alpha^i
 		+ \beta^3\sum_{j=n+1}^{T-1}\sum_{i=n+1}^{j}\theta_{j+1}^T\alpha^{2j-i-n} \\
 	&=&\frac{\beta^2\alpha^{-n-1}}{2}\sum_{i=n+1}^{T-1}\theta_i^T\alpha^i
 	+\beta^2\sum_{j=n+1}^{T-1}\theta_{j+1}^T\left(\sum_{i=n+1}^{j}\alpha^{2j-i-n+1}
 		-\sum_{i=n+1}^{j}\alpha^{2j-i-n}\right) \\
 	&=&\frac{\beta^2\alpha^{-n-1}}{2}\left((T-n)\alpha^{n+1}-\alpha^T\right)
 		+\beta^2\sum_{j=n+1}^{T-1}\theta_{j+1}^T\left(\alpha^{2j-2n}-\alpha^{j-n}\right) \textrm{ by Lemma \ref{simplelemma}},\\
 	&=&\frac{\beta^2(T-n)}{2}-\frac{\beta^2\alpha^{T-n-1}}{2}
 		+ \beta^2\sum_{j=n+1}^{T-1}\theta_{j+1}^T\alpha^{2j-2n}-\beta^2\sum_{j=n+2}^{T}\theta_{j}^T\alpha^{j-n-1} \\
 	&=&\frac{\beta^2(T-n)}{2}-\frac{\beta^2\alpha^{T-n-1}}{2}
 		+ \beta^2\sum_{j=n+1}^{T-1}\theta_{j+1}^T\alpha^{2j-2n}-\beta^2\alpha(T-n-1)\textrm{ by Lemma \ref{simplelemma}}. 
 	\end{eqnarray*}	
 		
 		The other terms in \eqref{sumofrows2} are
 		\begin{eqnarray*}
 		a_{1,n}&=&\frac{\beta\theta_n^T\alpha^{n-2}}{2}+\beta^2\sum_{j=n}^{T-1}\theta_{j+1}^T\alpha^{2j-n-1} \\
 		\beta a_{nn}&=&\frac{\beta}{2}+\beta^3\sum_{j=n}^{T-1}\theta_{j+1}^T\alpha^{2j-2n} \\
 		\beta a_{Tn}&=&\frac{\beta^2\alpha^{T-n-1}}{2}.
 		\end{eqnarray*}
 		
Substituting these into \eqref{sumofrows2}:
\begin{eqnarray*}
 	a_{1,n}-\beta\sum_{i=2}^Ta_{i,n}&=&\frac{\beta\theta_n^T\alpha^{n-2}}{2}+\beta^2\sum_{j=n}^{T-1}\theta_{j+1}^T\alpha^{2j-n-1} 
 		-\frac{\beta\theta_n^T\alpha^{n-2}}{2}+\frac{\beta\theta_n^T}{2}-\beta^2\sum_{j=n}^{T-1}\theta_{j+1}^T\alpha^{2j-n-1}\\
 	&\,&+\beta^2\sum_{j=n}^{T-1}\theta_{j+1}^T\alpha^{2j-2n+1}
 		 -\frac{\beta}{2}-\beta^3\sum_{j=n}^{T-1}\theta_{j+1}^T\alpha^{2j-2n}-\frac{\beta^2(T-n)}{2}+\frac{\beta^2\alpha^{T-n-1}}{2} \\
 	&\,&-\beta^2\sum_{j=n+1}^{T-1}\theta_{j+1}^T\alpha^{2j-2n}+\beta^2\alpha(T-n-1)-\frac{\beta^2\alpha^{T-n-1}}{2} \\
 	&=&\frac{\beta-\beta^2(T-n)}{2}+\beta^2\sum_{j=n}^{T-1}\theta_{j+1}^T\alpha^{2j-2n}-\frac{\beta}{2}-\frac{\beta^2(T-n)}{2} \\
  &\,& -\beta^2\sum_{j=n+1}^{T-1}\theta_{j+1}^T\alpha^{2j-2n}+\beta^2\alpha(T-n-1)\\
 	&=&\beta^2\theta_{n+1}^T-\beta^2(T-n)+ \beta^2\alpha(T-n-1)\\
 	&=&\beta^2(\theta_{n+1}^T-(T-n-\alpha(T-n-1)))=0.
\end{eqnarray*}
\end{proof}
 		
\begin{definition} For a matrix $\mathbf{A}\in\mathbb{R}^{n\times n}$  
let $\mathbf{A}(i,j)\in\mathbb{R}^{(n-1)\times (n-1)}$ denote the appropriate minor matrix of $\mathbf{A}$,
i.e. the matrix obtained by omitting the \emph{i}th row and the \emph{j}th column of $\mathbf{A}$.
\end{definition}
 	
\begin{lemma}\label{lemmaformatrices}		 
We have 
 		 \begin{equation} \mathbf{A}_T(1,1)=\mathbf{A}_{T-1}. \end{equation}
\end{lemma}
 		 
\begin{proof} We denote the elements of $\mathbf{A}_T(1,1)$ and $\mathbf{A}_{T-1}$ by $u_{i,k}$ and $v_{i,k}$, respectively.
 		 \begin{eqnarray*}
 		 u_{i,i}&=&a_{i+1,i+1}=\frac{1}{2}\beta^2\sum_{j=i+1}^{T-1}\theta_{j+1}^T\alpha^{2j-2(i+1)}
 		 	= \frac{1}{2}\beta^2\sum_{j=i}^{T-2}\theta_{j+2}^T\alpha^{2j-2i}\\
 		 &=&\frac{1}{2}\beta^2\sum_{j=i}^{T-2}\theta_{j+1}^{T-1}\alpha^{2j-2i}=v_{i,i} \qquad 1\leq i\leq T-2,\\
 		 u_{i,k}&=&a_{i+1,k+1}=\frac{\beta\theta_{i+1}^T\alpha^{i-k-1}}{2}+\beta^2\sum_{j=i+1}^{T-1}\theta_{j+1}^T\alpha^{2j-i-k-2}\\
 		 	&=&\frac{\beta\theta_{i}^{T-1}\alpha^{i-k-1}}{2}+\beta^2\sum_{j=i}^{T-1}\theta_{j+2}^T\alpha^{2j-i-k}\\
 		 	&=&\frac{\beta\theta_{i}^{T-1}\alpha^{i-k-1}}{2}+\beta^2\sum_{j=i}^{T-1}\theta_{j+1}^{T-1}\alpha^{2j-i-k}=v_{i,k} \qquad 1\leq k<i\leq T-2, \\
 		u_{T-1,k}&=&a_{T,k+1}=\frac{\beta\alpha^{T-(k+1)-1}}{2}=\frac{\beta\alpha^{(T-1)-k-1}}{2}=v_{T-1,k} \qquad 1\leq k\leq T-2\\
 		u_{T-1,T-1}&=&a_{TT}=\frac{1}{2}=v_{T-1,T-1}
 	\end{eqnarray*}
 	\end{proof}

\begin{lemma}\label{determinant1statement}
 		For the determinant of $\mathbf{A}_T$, for all $T\geq 2$, we have
 		\begin{equation}\label{determinant1}
 		\det \mathbf{A}_T=\frac{1+\beta^2(T-1)}{2}\det \mathbf{A}_{T-1}.
 		\end{equation}
\end{lemma}

\begin{proof} We construct a matrix $\mathbf{B}_T$ in such a way that we subtract the rows of $\mathbf{A}_T$ 
multiplied by $\beta$ from the first row. Then, according to Lemma \ref{sumofrowsstatement}, 
in the first row of $\mathbf{B}_T$ all elements expect the first one ($b_{1,1}$) are zero. 
Hence, using Lemma \ref{lemmaformatrices}

\[
\det \mathbf{B}_T=b_{1,1}\det \mathbf{B}_T(1,1)
=\left(a_{1,1}-\beta\sum_{i=2}^Ta_{i,1}\right)\det \mathbf{A}_T(1,1)
=\left(a_{1,1}-\beta\sum_{i=2}^Ta_{i,1}\right)\det \mathbf{A}_{T-1}.
\]

We need to check that 
\begin{equation}\label{firstelementofB}
a_{1,1}-\beta\sum_{i=2}^Ta_{i,1}=\frac{1+\beta^2(T-1)}{2}.
\end{equation}
Indeed,
 \begin{eqnarray*} \beta\sum_{i=2}^Ta_{i,1}&=&\beta\sum_{i=2}^{T-1}\left(\frac{\beta\theta_i^T\alpha^{i-2}}{2}
 +\beta^2\sum_{j=i}^{T-1}\theta_{j+1}^T\alpha^{2j-i-1} \right)+\frac{\beta^2\alpha^{T-2}}{2} \\
 &=&\frac{\beta^2\alpha^{-2}}{2}\sum_{i=2}^{T-1}\theta_i^T\alpha^i
 +\beta^3\sum_{j=2}^{T-1}\sum_{i=2}^j\theta_{j+1}^T\alpha^{2j-i-1}+ \frac{\beta^2\alpha^{T-2}}{2} \\
 &=&\frac{\beta^2\alpha^{-2}}{2}\left((T-1)\alpha^2-\alpha^T\right)+\beta^2\sum_{j=2}^{T-1}
 \theta_{j+1}^T\alpha^{2j-1} (\alpha-1)\sum_{i=2}^j\alpha^{-i}+ \frac{\beta^2\alpha^{T-2}}{2} \textrm{ by Lemma \ref{simplelemma}},\\
 &=&\frac{\beta^2(T-1)}{2}+\beta^2\sum_{j=2}^{T-1}\theta_{j+1}^T\alpha^{2j-1}\left(\sum_{i=2}^j\alpha^{-i+1}
 -\sum_{i=2}^j\alpha^{-i}\right)\\
 &=&\frac{\beta^2(T-1)}{2}+\beta^2\sum_{j=2}^{T-1}\theta_{j+1}^T\alpha^{2j-1}\left(\alpha^{-1}-\alpha^{-j}\right)\\
 &=&\frac{\beta^2(T-1)}{2}+\beta^2\sum_{j=2}^{T-1}\theta_{j+1}^T\alpha^{2j-2}-\beta^2\sum_{j=2}^{T-1}\theta_{j+1}^T\alpha^{j-1} \\
 &=&\frac{\beta^2(T-1)}{2}+\beta^2\sum_{j=2}^{T-1}\theta_{j+1}^T\alpha^{2j-2}-\beta^2\sum_{j=3}^{T}\theta_j^T\alpha^{j-2} \\
 &=&\frac{\beta^2(T-1)}{2}+\beta^2\sum_{j=2}^{T-1}\theta_{j+1}^T\alpha^{2j-2}-\beta^2(T-2)\alpha 
 \textrm{ by Lemma \ref{simplelemma}}.
 \end{eqnarray*}
 We substitute this into \eqref{firstelementofB}:
 
 \begin{eqnarray*}
 a_{1,1}-\beta\sum_{i=2}^Ta_{i,1}&=&\frac{1}{2}+\beta^2\sum_{j=1}^{T-1}\theta_{j+1}^T\alpha^{2j-2}- \frac{\beta^2(T-1)}{2}-\beta^2\sum_{j=2}^{T-1}\theta_{j+1}^T\alpha^{2j-2}+\beta^2(T-2)\alpha \\
 &=&\frac{1}{2}+\beta^2\theta_{2}^T-\frac{\beta^2(T-1)}{2}+\beta^2(T-2)\alpha \\
 &=&\frac{1}{2}+\beta^2(T-1-(T-2)\alpha)-\frac{\beta^2(T-1)}{2}+\beta^2(T-2)\alpha=\frac{1+\beta^2(T-1)}{2}.
 \end{eqnarray*}
 \end{proof}

 \begin{lemma}\label{determinant2statement}
 $\mathbf{A}_T$ is positive definite and its determinant is
 \begin{equation}\label{determinant2}
 \det \mathbf{A}_T=\frac{1}{2^T}\prod_{i=0}^{T-1}\left(1+\beta^2i\right)\qquad \textrm{for all }T\geq 1.
 \end{equation}
 \end{lemma}
 
\begin{proof} For $T=1$, \eqref{determinant2} gives $1/2$. During the computation of $\hat{\phi}_1^1$ we 
saw that the coefficient of the quadratic term was indeed $1/2$. Let's assume that \eqref{determinant2} holds for $T-1$, namely
 \begin{equation}\label{assumption}
 \det \mathbf{A}_{T-1}=\frac{1}{2^{T-1}}\prod_{i=0}^{T-2}\left(1+\beta^2i\right).
 \end{equation}
 Then
 \begin{eqnarray*}\det \mathbf{A}_T&=&\frac{1+\beta^2(T-1)}{2}\det \mathbf{A}_{T-1},
 \quad\textrm{by Lemma \ref{determinant1statement},}\\
 		&=&\frac{1+\beta^2(T-1)}{2}\frac{1}{2^{T-1}}\prod_{i=0}^{T-2}\left(1+\beta^2i\right), 
 		\quad\textrm{by \eqref{assumption},}\\ 
 		&=&\frac{1}{2^T}\prod_{i=0}^{T-1}\left(1+\beta^2i\right).
 \end{eqnarray*}
 
 	Since $\det \mathbf{A}_{T}>0$ for all $T\geq 1$, Lemma \ref{lemmaformatrices} applies and 
 	the determinants of the matrices $[a_{ij}]_{i=n,\ldots,T;j=n,\ldots,T}$ are positive for all
 	$1\leq n\leq T$, 
 	therefore $\mathbf{A}_{T}$ is positive definite for all $T\geq 1$.
 \end{proof}
 	
Obviously, we can express $\det \mathbf{A}_T$ with the well-known $\Gamma$ function.
\begin{lemma}\label{determinantwithgamma}
We have $\det\mathbf{A}_T=\gamma_\beta(T)/2^T$, where $\gamma_\beta$ is the function defined in \ref{gamma}. 
\hfill$\Box$\end{lemma}



 
 Later on we will need the value of $p:=\left(\mathbf{A}_{T}^{-1}\right)_{1,1}$. 
 Now we compute it using Lemmas \ref{lemmaformatrices} and \ref{determinant1statement}:
 
 \begin{equation}\label{inverse}
 p=\frac{\det \mathbf{A}_T(1,1)}{\det \mathbf{A}_T}=\frac{\det \mathbf{A}_{T-1}}{\det \mathbf{A}_T}=\frac{2}{1+\beta^2(T-1)}.
 \end{equation}


We need to  compute the minimiser of \eqref{matrixformofexp}. Note that in \eqref{matrixformofexp} only the exponent depends on $\phi$, so we can focus on this. Let 
\begin{equation}\label{f}
f(\phi, z):=\frac{\mathbf{b^T}(\phi,z)\mathbf{A}_T^{-1}\mathbf{b}(\phi,z)}{4}-c(\phi,z).
\end{equation}

Then, we need to solve
\begin{equation}\label{diffoff}
\partial_{\phi}f(\phi,z)=\frac{\partial_{\phi}\mathbf{b^T}(\phi,z)\mathbf{A}_T^{-1}\mathbf{b}(\phi,z)}{2}-\partial_{\phi}c(\phi,z)=0.
\end{equation}

From the definition a $\mathbf{b}(\phi,z)$ and $c(\phi,z)$
\begin{eqnarray}
\partial_{\phi}\mathbf{b^T}(\phi,z)&=&\left(\sigma,0,\dots,0\right), \label{diffofb}\\
\partial_{\phi}c(\phi,z)&=&\beta z. \label{diffofc}
\end{eqnarray}
 
 Note that we can write $\mathbf{b}(\phi,z)$ as
 
 \begin{equation}\label{b}
 \mathbf{b}(\phi,z)=\left(\sigma\phi-\frac{\alpha z}{\sigma}\right)\mathbf{e_1}+\frac{2\alpha z}{\sigma}\mathbf{A}_T(:,1),
 \end{equation}
 
 where $\mathbf{e_1}=\left(1,0\dots,0\right)^{\mathbf{T}}$, 
 and $\mathbf{A}_T(:,1)$ is the first column of $\mathbf{A}_T$. 
 Let's substitute \eqref{diffofb}, \eqref{diffofc} and \eqref{b} into \eqref{diffoff}: 		

\begin{eqnarray*}
	(\sigma, 0,\dots,0)\mathbf{A}_T^{-1}\left[ \left(\sigma\phi-\frac{\alpha z}{\sigma}\right)\mathbf{e_1}+
		\frac{2\alpha z}{\sigma}\mathbf{A}_T(:,1)\right]-2\beta z&=&0 \\
	\sigma\mathbf{A}_T^{-1}(1,:)\left[\left(\sigma\phi-\frac{\alpha z}{\sigma}\right)\mathbf{e_1}+
		\frac{2\alpha z}{\sigma}\mathbf{A}_T(:,1)\right]-2\beta s_0&=&0 \\
	\left(\sigma^2\phi-2\alpha z\right)p+2\alpha z-2\beta z&=&0 \\
	\sigma^2\phi-\alpha z&=&-\frac{2z}{p} \\
	\sigma^2\phi-\alpha z&=&-\left(1+\beta^2(T-1)\right)z\quad\textrm{by \eqref{inverse}, hence} \\
	\phi&=&\frac{\beta z}{\sigma^2}\left(1-(T-1)\beta\right) \\
	\phi&=&\frac{\beta z}{\sigma^2}\theta_1^T.
\end{eqnarray*}

	We can see from the above calculation that this $\phi$ is a global minimiser of $f$ for a given $z$.
	Hence the minimiser $\phi$ for \eqref{phi1def} is
	\[\phi=\hat{\phi}_1^T(z)=\frac{\beta z}{\sigma^2}\theta_1^T,\]
	and we have proved the first part \autoref{maintheorem} in the case of using past information. 
	As we have found explicit optimal strategies for the expected utility problem, we can now turn to \eqref{condexp} 
	and \eqref{exp}.


First we compute the maximal conditional expected utility which is
\begin{equation}\label{almostdone}
\mathbb{E}\left[U\left(L_T^{{\phi}^*}\right)|X_0=z\right]=-\frac{1}{\sqrt{\gamma_\beta(T)}}e^{f\left(\hat{\phi}_1^T(z),z\right)}
\quad\textrm{by \eqref{matrixformofexp},  Lemma \ref{determinantwithgamma} and \eqref{f}}.
\end{equation}

Let
$\hat{f}(z):=f\left(\hat{\phi}_1^T(z),z\right)$,
$\mathbf{\hat{b}}(z):=\mathbf{b}\left(\hat{\phi}_1^T(z),z\right)$ and
$\hat{c}(z):=c\left(\hat{\phi}_1^T(z),z\right)$.

First we compute $\mathbf{\hat{b}^T}(z)\mathbf{A}_T^{-1}\mathbf{\hat{b}}(z)$:

\begin{eqnarray*}
	\mathbf{\hat{b}^T}(z)\mathbf{A}_T^{-1}\mathbf{\hat{b}}(z)&=& \left[\frac{2\alpha z}{\sigma}\mathbf{A}_T(:,1)
	-\frac{z}{\sigma}\left(1+(T-1)\beta^2\right)\mathbf{e_1}\right]^{\mathbf{T}}\mathbf{A}_T^{-1}\\
	&\,&\cdot\left[\frac{2\alpha z}{\sigma}\mathbf{A}_T(:,1)-\frac{z}{\sigma}\left(1+(T-1)\beta^2\right)\mathbf{e_1}\right]
	 \quad\textrm{by \eqref{b},}\\
&=&\left[\frac{2\alpha z}{\sigma}\left(1,0,\dots,0\right)-\frac{z}{\sigma}\left(1+(T-1)\beta^2\right)\mathbf{A}_T^{-1}(1,:)\right] 
\cdot\left[\frac{2\alpha z}{\sigma}\mathbf{A}_T(:,1)-\frac{z}{\sigma}\left(1+(T-1)\beta^2\right)\mathbf{e_1}\right],\\
&=&\frac{z^2}{\sigma^2}\left(\left(1+\beta^2(T-1)\right)^2p-4\alpha\left(1+\beta^2(T-1)\right)+4\alpha^2a_{1,1}\right)\\
&=&\frac{z^2}{\sigma^2}\left(2\left(1+\beta^2(T-1)\right)-4\alpha\left(1+\beta^2(T-1)\right)+4\alpha^2a_{1,1}\right)
\quad\textrm{by \eqref{inverse}.}
\end{eqnarray*}

 For $\hat{f}(z)$
 \begin{eqnarray*}
 \hat{f}(z)&=&\frac{\mathbf{\hat{b}^T}(z)\mathbf{A}_T^{-1}\mathbf{\hat{b}}(z)}{4}-\hat{c}(z)\\
  &=&\frac{z^2}{\sigma^2}\left(\frac{\left(1+\beta^2(T-1)\right)}{2}-\alpha\left(1+\beta^2(T-1)\right)
  +\alpha^2a_{1,1}- \beta(1-(T-1)\beta)-\beta^2\sum_{j=2}^T\theta_j^T\alpha^{2j-2}\right)\\
  &=&\frac{z^2}{\sigma^2}\left(\frac{\left(1+\beta^2(T-1)\right)}{2}-\alpha\left(1+\beta^2(T-1)\right)
  +\frac{\alpha^2}{2}+\beta^2\sum_{j=1}^{T-1}\theta_{j+1}^T\alpha^{2j}
  -\beta(1-(T-1)\beta)-\beta^2\sum_{j=2}^T\theta_j^T\alpha^{2j-2}\right)\\
  &=&-\frac{\beta^2 z^2}{2\sigma^2}T.
 \end{eqnarray*}
 
 Hence the maximal achievable conditional expected utility is
 \begin{equation}
 \mathbb{E}\left[U\left(L_T^{{\phi}^*}\right)|X_0=z\right]=-\frac{1}{\sqrt{\gamma_\beta(T)}}e^{-\frac{\beta^2z^2}{2\sigma^2}T},
 \end{equation}
 and we have proved \eqref{condexp}. Now we prove \eqref{exp}. For stable processes, in case of $\mathrm{var}(X_t)=1$ 
 for all $t\geq 0$
 
 \begin{equation} 1=\mathrm{var}(X_t)=\mathrm{var}\left(\sigma\sum_{j=0}^{+\infty}\alpha^j\varepsilon_{t-j}\right)=\sigma^2\sum_{j=0}^{+\infty}\alpha^{2j}=\frac{\sigma^2}{1-\alpha^2},\nonumber
 \end{equation}
 therefore $\sigma^2=1-\alpha^2$. As $X_0$ is $N(0,1)$, the maximal expected utility can be found using 
 \eqref{integral1}:
 \begin{eqnarray*}
 \mathbb{E}\left[U\left(L_T^{{\phi^*}}\right)\right]&=&
 \mathbb{E}\left[\mathbb{E}\left[U\left(L_T^{{\phi}^*}\right)|X_0 \right]\right] 
  =\mathbb{E}\left[-\frac{1}{\sqrt{\gamma_\beta(T)}}e^{-\frac{\beta^2X_0^2}{2\sigma^2}T}\right]\\
	&=&-\frac{1}{\sqrt{2\pi \gamma_\beta(T)}}\int_{\mathbb{R}}e^{-\frac{\beta^2x^2}{2\sigma^2}T-\frac{x^2}{2}}\mathrm{d}x 
	=-\frac{1}{\sqrt{\gamma_\beta(T)\left(\frac{\beta^2T}{1-\alpha^2}+1\right)}} \\
	&=&-\sqrt{\frac{\beta+2}{\left(2-(T-1)\beta\right)\gamma_\beta(T)}},
	\end{eqnarray*}
so we have proved \eqref{exp}.


Now we focus on the case where the strategies depend only on the initial value $X_0$
of the autoregressive process. In this case using the strategy $\eta=(\eta_1,\dots,\eta_T)$  we get
\begin{eqnarray*}
L_T^\eta&=&\sum_{j=1}^T\eta_j\left(X_j-X_{j-1}\right) \\
		&=&\sum_{j=1}^T\eta_j\left(\alpha^{j-1}\beta X_0+\sigma\varepsilon_j+\sigma\beta\sum_{k=1}^{j-1}\alpha^{j-1-k}\varepsilon_k\right) \\
		&=&\beta X_0\sum_{j=1}^T\eta_j\alpha_{j-1}+\sigma\sum_{j=1}^T\eta_j\varepsilon_j+\beta\sigma\sum_{j=1}^T\sum_{k=1}^{j-1}\eta_j \alpha^{j-k-1}\varepsilon_k \\
		&=&\beta X_0\sum_{j=1}^T\eta_j\alpha_{j-1}+\sigma\sum_{j=1}^T\eta_j\varepsilon_j +\beta\sigma\sum_{k=1}^T\varepsilon_k\sum_{j=k+1}^T\eta_j\alpha^{j-k-1}.
	\end{eqnarray*}
	
	Let $c: \mathbb{R}^T\rightarrow\mathbb{R}$, and $\mathbf{b}:\mathbb{R}^T\rightarrow\mathbb{R}^T$, where
	\[c(\eta)=\beta X_0\sum_{j=1}^T\eta_j\alpha^{j-1},\quad
	\textrm{and}\quad b_k(\eta)=\sigma\varepsilon_k+\beta\sigma\sum_{j=k+1}^T\eta_j\alpha^{j-k-1},\quad 1\leq k\leq T.\]
Using the notation $L_T^\eta=c(\eta)+\mathbf{b^T}(\eta)\mathbf{\varepsilon}$ we get from 
Lemma \ref{integral2lemma} that
\begin{equation}\label{almostdonenomem}
\mathbb{E}\left[U\left(L_T^\eta\right)|X_0\right]=-e^{g(\eta)}, \textrm{ where } 
g(\eta)=-c(\eta)+\frac{\mathbf{b^T}(\eta)\mathbf{b}(\eta)}{2}.
\end{equation} 

We need to solve the system of equations $\nabla g(\eta)=\mathbf{0}$. 
We denote these equations by $(E_k)$, where $1\leq k\leq T$:
\[0=\frac{\partial g}{\partial \eta_k}(\eta)=-X_0\beta\alpha^{k-1}+\mathbf{b^T}(\eta)\frac{\partial\mathbf{b}}{\partial\eta_k}(\eta).
\quad (E_k)\]
The partial derivatives of $\mathbf{b}$ are:
\[\left(\frac{\partial\mathbf{b}}{\partial\eta_k}(\eta)\right)_l=\left\{\begin{array}{cl} \sigma\beta\alpha^{k-l-1} & \textrm{if }l<k \\
																																										\sigma & \textrm{if } l=k \\
																																										0 & \textrm{if } l>k. \end{array} \right. \]
Hence we have
\begin{eqnarray*}
\mathbf{b^T}(\eta)\frac{\partial\mathbf{b}}{\partial\eta_k}(\eta)&=&\sigma\beta\sum_{l=1}^{k-1}\alpha^{k-l-1}\left(\sigma\eta_l +\sigma\beta\sum_{j=l+1}^T\eta_j\alpha^{j-l-1}\right)+\sigma^2\eta_k+\sigma^2\beta\sum_{j=k+1}^T\eta_j\alpha^{j-k-1} \\
	&=&\sigma^2\beta\sum_{l=1}^{k-1}\eta_l\alpha^{k-l-1}+\sigma^2\beta^2\sum_{l=1}^{k-1}\sum_{j=l+1}^T\eta_j\alpha^{j+k-2l-2} +\sigma^2\eta_k+\sigma^2\beta\sum_{j=k+1}^T\eta_j\alpha^{j-k-1}.
	\end{eqnarray*}
	
Therefore the equations $E_k$ take the form 
\[\beta\sum_{l=1}^{k-1}\eta_l\alpha^{k-l-1}+\beta^2\sum_{l=1}^{k-1}\sum_{j=l+1}^T\eta_j\alpha^{j+k-2l-2} 
+\eta_k+\beta\sum_{j=k+1}^T\eta_j\alpha^{j-k-1}=\frac{X_0}{\sigma^2}\beta\alpha^{k-1}. \]

Let $k\in\{1, 2, \dots ,T-1\}$. Then for $E_{k+1}$ we have
\[\beta\sum_{l=1}^{k}\eta_l\alpha^{k-l}+\beta^2\sum_{l=1}^{k}\sum_{j=l+1}^T\eta_j\alpha^{j+k-2l-1} +\eta_{k+1}+\beta\sum_{j=k+2}^T\eta_j\alpha^{j-k-2}=\frac{X_0}{\sigma^2}\beta\alpha^{k}. \]
We define equation $F_k$ for $1\leq k\leq T-1$ by substract equation $E_k$ multiplied by $\alpha$ from equation $E_{k+1}$, so we get
\begin{eqnarray*}
0&=&\beta\eta_k+\beta^2\sum_{j=k+1}^T\eta_j\alpha^{j-k-1}+\eta_{k+1}-\alpha\eta_k
+\beta\sum_{j=k+2}^T\eta_j\alpha^{j-k-2} -\beta\sum_{j=k+1}^T\eta_k\alpha^{j-k}, \\
0&=&\eta_{k+1}-\eta_k+\beta\left(\sum_{j=k+1}^T\eta_j\alpha^{j-k}-\sum_{j=k+1}^T\eta_j\alpha^{j-k-1}\right)
+\beta\sum_{j=k+2}^T\eta_j\alpha^{j-k-2} -\beta\sum_{j=k+1}^T\eta_k\alpha^{j-k}, \\
0&=&\eta_{k+1}-\eta_k-\beta\sum_{j=k+1}^T\eta_j\alpha^{j-k-1}+\beta\sum_{j=k+2}^T\eta_j\alpha^{j-k-2}, \\
0&=&\eta_{k+1}-\eta_k-\beta\eta_{k+1}+\beta\sum_{j=k+2}^T\eta_j\alpha^{j-k-2}(1-\alpha), \\
0&=&(1-\beta)\eta_{k+1}-\eta_k-\beta^2\sum_{j=k+2}^T\eta_j\alpha^{j-k-2}. \\
\end{eqnarray*}

\begin{lemma}\label{eqnstatement} 
For the solutions of the system $F_k$, $k=1,\ldots,T-1$,
\begin{equation}\label{eqn}
 \eta_k=\theta_k^T\eta_T=(1-(T-k)\beta)\eta_T 
\end{equation}
hold, for all $k=1,\ldots, T-1$.
\end{lemma}

\begin{proof} First we consider the equation $F_{t-1}$,
\begin{eqnarray*}
0&=&(1-\beta)\eta_T-\eta_{T-1} \\
\eta_{T-1}&=&\theta_{T-1}^T\eta_T
\end{eqnarray*}

Let's assume that \eqref{eqn} holds for $l=k+1, \dots ,T-1$. 
Considering the equation $F_k$, using Lemma \ref{simplelemma} we get
\begin{eqnarray*}
0&=&(1-\beta)\eta_{k+1}-\eta_k-\beta^2\sum_{j=k+2}^T\eta_j\alpha^{j-k-2}, \\
0&=&(1-\beta)(1-(T-k-1)\beta)\eta_T-\eta_k-\beta^2\sum_{j=k+2}^T\theta_j^T\eta_T\alpha^{j-k-2}, \\
0&=&(1-\beta)(1-(T-k-1)\beta)\eta_T-\eta_k-\beta^2(T-k-1)\eta_T, \\
\eta_k&=&(1-(T-k-1)\beta-\beta)\eta_T ,\\
\eta_k&=&\theta_k^T\eta_T.
\end{eqnarray*}
\end{proof}

Because of $\theta_T^T=1$, $\eta_T=\theta_T^T\eta_T$ also holds. To prove the first part of 
\autoref{maintheorem} for the case without using past information, we only need to show that $\eta_T=\beta\frac{X_0}{\sigma^2}$.

Substituting \eqref{eqn} into $E_T$,
\begin{equation}
\beta\sum_{l=1}^{T-1}\theta_l\eta_T\alpha^{T-l-1}+\beta^2\sum_{l=1}^{T-1}\sum_{j=l+1}^T\theta_j^T\eta_T\alpha^{j+T-2l-2} 
+\eta_T=\frac{X_0}{\sigma^2}\beta\alpha^{T-1} \\
\end{equation}
We compute the two sums separately,
\begin{eqnarray*}
\sum_{l=1}^{T-1}\theta_l^T\alpha^{T-l-1}&=&\sum_{l=1}^{T-1}\left(T+1-l-(T-l)\alpha\right)\alpha^{T-l-1} 
	=\sum_{l=1}^{T-1}(T+1-l)\alpha^{T-l-1}-\sum_{l=1}^{T-1}(T-l)\alpha^{T-l} \\
	&=&\sum_{l=2}^{T}(T+2-l)\alpha^{T-l}-\sum_{l=1}^{T-1}(T-l)\alpha{T-l}
	=\sum_{l=2}^T2\alpha^{T-l}-(T-1)\alpha^{T-1}; \\
\beta\sum_{l=1}^{T-1}\sum_{j=l+1}^T\theta_j^T\alpha^{j+T-2l-2}&=&\beta\sum_{l=1}^{T-1}\alpha^{T-2l-2}\sum_{j=l+1}^T\theta_j^T\alpha^j
	=\beta\sum_{l=1}^{T-1}(T-l)\alpha^{T-l-1} \\
	&=&\sum_{l=1}^{T-1}(T-l)\alpha^{T-l}-\sum_{l=1}^{T-1}(T-l)\alpha^{T-l-1}
	=\sum_{l=1}^{T-1}(T-l)\alpha^{T-l}-\sum_{l=2}^{T}(T-l+1)\alpha^{T-l}\\
	&=&-\sum_{l=2}^T\alpha^{T-l}+(T-1)\alpha^{T-1}.
\end{eqnarray*}

Hence, from $E_T$:
\begin{eqnarray*}
\left(1+\beta\sum_{l=2}^T\alpha^{T-l}\right)\eta_T&=&\alpha^{T-l}\beta\frac{X_0}{\sigma^2} \\
\left(1+\sum_{l=2}^T\alpha^{T-l+1}-\sum_{l=2}^T\alpha^{T-l}\right)\eta_T&=&\alpha^{T-l}\beta\frac{X_0}{\sigma^2} \\
\alpha^{T-1}\eta_T&=&\alpha^{T-l}\beta\frac{X_0}{\sigma^2} \\
\eta_T&=&\beta\frac{X_0}{\sigma^2}.
\end{eqnarray*}
So we have proved the firs part of \autoref{maintheorem}. Now we prove \eqref{condexpnomem}. First we compute 
the element of  $\hat{\mathbf{b}}:=\mathbf{b}(\hat{\eta})$, using Lemma \ref{simplelemma},
\begin{equation}
\hat{b}_k=\frac{\beta X_0}{\sigma}\theta_k+\frac{\beta^2 X_0}{\sigma}\sum_{j=k+1}^T\theta_j\alpha^{j-k-1}
	=\frac{\beta X_0}{\sigma}(1-(T-k)\beta)+\frac{\beta^2 X_0}{\sigma}(T-k)
	=\frac{\beta X_0}{\sigma}.
\end{equation}
Therefore 
$\mathbf{\hat{b}^T\hat{b}}=\frac{\beta^2 X_0^2}{\sigma^2}T$, and
$\hat{c}:=c(\hat{\eta})=\frac{\beta^2X_0^2}{\sigma^2}\sum_{j=1}^T\theta_j\alpha^{j-1}=\frac{\beta^2X_0^2}{\sigma^2}T$, by 
Lemma \ref{simplelemma}.
Hence, using \eqref{almostdonenomem}, we get
$\mathbb{E}\left[U\left(L_T^{\hat{\eta}}\right)|X_0\right]
=e^{-\frac{\beta^2X_0^2}{2\sigma^2}T}$, which proves \eqref{condexpnomem}. 
Based on the same calculation that we did to get \eqref{exp} from \eqref{condexp},
we get \eqref{expnomem} from \eqref{condexpnomem}.


The only statment of \autoref{maintheorem} left to prove is the forth part. For $\beta=0$, the statement is trivial. For the other case, using $y(T):=T-1+\frac{1}{\beta^2}$, we get
\begin{equation}
\frac{\gamma_\beta(T)}{h_\beta(T)}=\frac{\beta^{2T}\Gamma(y(T)+1)}
{\left(\beta^2\right)^{1-\frac{1}{\beta^2}}\sqrt{2\pi y(T)}\left(\frac{\beta^2y(T)}{e}\right)^{T-1+\frac{1}{\beta^2}}}
=\frac{\Gamma\left(y(T)+1\right)}{\sqrt{2\pi y(T)}\left(\frac{y(T)}{e}\right)^{y(T)}},
\end{equation}
and it is well-known  that this expression tends to $1$ if $T$ (and hence also $y(T)$) tend to infinity.


\begin{thebibliography}{9}

\bibitem{baillie} R.~T. Baillie. Long memory processes and fractional integration in econometrics.
 {\em J. Econometrics}, 73:5--59, 1996.

\bibitem{dm}
Dellacherie, C. and Meyer, P. A. {\em Probabilities and potential.}
Mathematical studies {29}, North-Holland, Amsterdam, 1978.

 
 \bibitem{dokuchaev} N. Dokuchaev. Mean-reverting market models: speculative opportunities and non-arbitrage. 
 \emph{Appl. Math. Finance}, 14:319--337, 2007.

 \bibitem{fs} H. F\"ollmer and W. Schachermayer. 
 Asymptotic arbitrage and large deviations. \emph{Math. Financ. Econ.}, Vol. 1:213--249, 2007.

 \bibitem{martin} M. L. D. Mbele-Bidima and M. R\'asonyi.
 Asymptotic exponential arbitrage and utility-based asymptotic arbitrage in Markovian models of financial 
markets. {\em To appear in Acta Applicandae Mathematicae}, 2014.\texttt{ arXiv:1406.5312}

\end{thebibliography}
\end{document}